\documentclass[11pt]{amsart}
\usepackage{amsmath}
\usepackage{amsfonts}
\usepackage{amssymb,enumerate}
\usepackage{amsthm}
\usepackage{hyperref}
\usepackage{centernot}
\usepackage{stmaryrd}

\theoremstyle{plain}
\newtheorem{theorem}{Theorem}[section]

\newtheorem{proposition}[theorem]{Proposition}

\begin{document}

\title{On the infinitary van der Waerden Theorem}
\author{Shahram Mohsenipour}

\address{School of Mathematics, Institute for Research in Fundamental Sciences (IPM)
        P. O. Box 19395-5746, Tehran, Iran}\email{mohseni@ipm.ir}


\subjclass [2010] {05D10}
\keywords{Infinitary van der Waerden's theorem}
\begin{abstract} We give a purely combinatorial proof for the infinitary van der Waerden's theorem.
\end{abstract}
\maketitle
\bibliographystyle{amsplain}
\section{Introduction}
In \cite{thomasshahram}, Thomas Brown and the author introduced the following infinitary generalization of van der Waerden's theorem which has the same relation with van der Waerden's theorem as the Carlson-Simpson theorem \cite{carlsonsimpson} has with the Halse-Jewett theorem \cite{hj}.\\

{\em \textbf{Theorem}. Let $k\geq 2$ be a positive integer. Then, for any finite coloring of $\mathbb{N}$, there exist positive integers $d_1,d_2,d_3,\dots$ and a color $\gamma$ such that for every $n\in\mathbb{N}$, there exists a positive integer $a_n$ such that the set
\[
\big{\{}a_n+j_1d_1+\cdots+j_nd_n|0\leq j_1,\dots,j_n\leq k-1\big{\}}
\]
is monochromatic with the color $\gamma$.}\\


It is possible to deduce the above theorem from the Carlson-Simpson theorem by generalizing the usual proof of deducing van der Waerden's theorem from the Hales-Jewett theorem. The proof given in \cite{thomasshahram} uses a more elementary fact on uniformly recurrent infinite words from \cite{justinpirillo} which has its root in symbolic dynamics. It is worth noting that the above theorem can also be easily deduced from the Furstenberg-Weiss Theorem (\cite{furstenbergweiss}, Theorem 1.5) by successive iteration. In this note we present a direct combinatorial proof.

For the convenience of the reader and at the suggestion of the referee, we show in the next section how the above theorem can be related to the Carlson-Simpson theorem and then in Section 3, we present our elementary proof. Now, we explain some notation and conventions that we will use in the sequel. For two finite subsets $A,B$ of $\mathbb{N}$, by $A<B$, we mean that every element of $B$ is greater than every element of $A$. By interval, we mean an interval of positive integers which is a set of the form $I=\{a,a+1,a+2,\dots,b\}$, for some positive integers $a\leq b$. We also denote $I$ by $[a,b]$, and for $1\leq i\leq b-a+1$, by $I[i]$ we mean the $i$th element of the interval $[a,b]$, namely, $a+i-1$. The number of elements of $I$ is denoted by $|I|$. If $b$ is a non-negative integer and $A$ is a set of integers, then $\mathbb{T}^b A$ will denote the set $\{a+b|a\in A\}$. Let $c$ be a positive integer, by a $c$-coloring, we mean a coloring with $c$ many colors and furthermore the set of colors is $[c]=\{1,\dots,c\}$. For positive integers $k\geq 2$, and $c\geq 1$, $W(k,c)$ will denote the van der Waerden number which is the least integer $n$ such that for any $c$-coloring of $[n]$, there exists a monochromatic arithmetic progression of length $k$.

\section{Arithmetic Progressions and the Carlson-Simpson Theorem}

The original van der Waerden's theorem says that for a given $k\geq2$ and any finite coloring of the set of positive integers, $\mathbb{N}$, there exists a monochromatic arithmetic progression of length $k$, namely, there are $a,d\in\mathbb{N}$ such that $a,a+d,\dots,a+(k-1)d$, all have the same color. It is easy to see that the obvious generalization to the infinite case is not correct and in fact there is a 2-coloring of $\mathbb{N}$ such that for no $a,d\in\mathbb{N}$, $a,a+d,a+2d,a+3d,\dots$, all have the same color, in other words, there is no monochromatic arithmetic progression of infinite length.

Now recall that van der Waerden's theorem can be deduced from the Hales-Jewett theorem, which itself has an infinite generalization, i.e., the Carlson-Simpson theorem. So in order to obtain a suitable explicit infinite generalization of van der Waerden's theorem, one might ask: What does the Carlson-Simpson theorem say about arithmetic progressions? It is easily seen that the answer is exactly what we stated above as the infinite generalization of van der Waerden's theorem. To see this, we first bring here the statements of the above mentioned theorems. We follow Walters' paper \cite{walters1}. Let $[k]=\{1,2,\dots,k\}$ and let $[k]^N$ be the set of all sequences $a_1a_2\dots a_N$ of length $N$ such that $a_i\in[k]$ for $i=1,\dots,N$. For $a\in[k]^N, \gamma\subseteq\{1,\dots,N\}$ and $x\in[k]$, let $a\oplus x\gamma$ denote the sequence $b\in[k]^{N}$ defined by $a_i=b_i$ if $i\notin\gamma$ and $b_i=x$ otherwise. A {\em combinatorial line} is a set of the form $\{a\oplus x\gamma |1\leq x\leq k\}$. The set $\gamma$ is called the set of active coordinates. Now we can state the Hales-Jewett theorem:

\begin{theorem}[Hales-Jewett]
For $k,c$ there exists $N$ such that whenever $[k]^N$ is $c$-colored, there exist $a\in[k]^N$ and $\gamma\subseteq[N]$ such that the set $\{a\oplus x\gamma |1\leq x\leq k\}$ is monochromatic.
\end{theorem}

Now observe that the following mapping
\[
a_1a_2\dots a_N\mapsto (a_1-1)+(a_2-1)k+\cdots+(a_N-1)k^{N-1}+1
\]
gives a bijection $[k]^N\rightarrow\{1,2,\dots,k^N\}$ and the image of the monochromatic line $\{a\oplus x\gamma |1\leq x\leq k\}$ is the arithmetic progression
\[
\Big{\{}\sum_{i\notin\gamma}(a_i-1)k^{i-1}+x(\sum_{i\in\gamma}k^{i-1})+1 | 0\leq x\leq k-1\Big{\}}.
\]

Now define a {\em word} on an {\em alphabet} $[k]$ to be an element of $[k]^N$ for some $N$. For $N_2>N_1$, we say that a word $a^2\in[k]^{N_2}$ extends a word $a^1\in[k]^{N_1}$ and denote it by $a^1<_{e}a^2$, if $a^2_i=a^1_i$ for $i=1,\dots,N_1$. Let $Q(k)$ denote the set of all words on alphabet $[k]$.

\begin{theorem}[Carlson-Simpson]
Whenever $Q=Q(k)$ is finitely colored, there exist an infinite strictly increasing sequence of positive integers $N_n$, an infinite sequence of words $a^1<_{e}a^2<_{e}a^3<_{e}\cdots$, where $a^n\in [N_n]^k$ and an infinite sequence of finite sets $\gamma_1<\gamma_2<\gamma_{3}<\cdots$, where $\gamma_n\subseteq[N_n]$ such that
the set
\[
\big{\{}a^n\oplus x_1\gamma_1\oplus\cdots\oplus x_n\gamma_n | 1\leq x_1,\dots x_n\leq k, n\in\mathbb{N}\big{\}}
\]
is monochromatic.
\end{theorem}
We just mention here the case $n=2$ and leave the full case to the reader. In this case it is enough to see that the image of
\[
\big{\{}a^2\oplus x_1\gamma_1\oplus x_2\gamma_2 | 1\leq x_1, x_2\leq k\big{\}}
\]
under the above mentioned mapping is
\[
\Big{\{}\sum_{i\notin\gamma_1\cup\gamma_2}(a^2_i-1)k^{i-1}+x_1(\sum_{i\in\gamma_1}k^{i-1})+x_2(\sum_{i\in\gamma_2}k^{i-1})+1 | 0\leq x_1,x_2\leq k-1\Big{\}}.
\]

\section{Infinitary van der Waerden's Theorem}

Fix positive integers $k\geq 2$, and $c\geq 1$. We inductively define sequences of positive integers $\langle W_n:n\in\mathbb{N}\rangle$ and $\langle c_n:n\in\mathbb{N}\rangle$ (depending on $k,c$) as follows. Put $W_1=W(k,c), W_{n+1}=W(k,c_n)$ and $c_n=c^{W_n\cdots W_1}$. Now let $I$ be an interval with $|I|=W_1$. Now from $I$, we inductively define a sequence of intervals $I_n, n\in\mathbb{N}$ as follows. Let $I_1=I$ and
\[
I_{n+1}=I_n\cup \mathbb{T}^{W_n\cdots W_1}I_n\cup \mathbb{T}^{2W_n\cdots W_1}I_n\cup\cdots\cup\mathbb{T}^{(W_{n+1}-1)W_n\cdots W_1}I_n.
\]
For instance for $n=2$, we have
\[
I_{2}=I_1\cup \mathbb{T}^{W_1}I_1\cup \mathbb{T}^{2W_1}I_1\cup\cdots\cup\mathbb{T}^{(W_{2}-1)W_1}I_1,
\]
which according to $|I_1|=W_1$, it is easily seen that $I_2$ is an interval with $|I_2|=W_2W_1$. Similarly for each $n\in\mathbb{N}$, we have that $I_n$ is an interval with $|I_n|=W_n\cdots W_1$. Obviously in the above definition, $I_n$ depends on $k,c,n,I$, so we denote it by $I_n=\mathcal{I}(I,n,k,c)$. If $k,c$ were fixed and were clear from the context, then we may denote it by $I_n=\mathcal{I}(I,n)$. We will also use the easy fact that for any positive integer $b$, the interval $\mathbb{T}^b \mathcal{I}(I,n,k,c)$ can be written in the form of $\mathcal{I}(J,n,k,c)$ for some interval $J$ with $|J|=W_1$. In fact due to the additive nature of the definition, we have $\mathbb{T}^b \mathcal{I}(I,n,k,c)=\mathcal{I}(\mathbb{T}^{b}I,n,k,c)$.

\begin{proposition}\label{first}
Let $k\geq 2$, $c\geq 1$, and $n\geq 1$ be positive integers and let $I$ be an interval with $|I|=W_1$. Then for any $c$-coloring of $\mathcal{I}(I,n)$, there exist $a, d_1,\dots, d_n$ such that $d_i\leq W_1\cdots W_i$ for $i=1,\dots,n$ and the following set
\[
\big{\{}a+j_1d_1+\cdots+j_nd_n|0\leq j_1,\dots,j_n\leq k-1\big{\}}
\]
is included in $\mathcal{I}(I,n)$ and is monochromatic.
\end{proposition}
\begin{proof}
The proof is by induction on $n$. The case $n=1$ is just van der Waerden's theorem and $d_1\leq W_1$ is obvious. Now, assume that the assertion is true for $n$, we prove it for $n+1$. Let $\chi\colon\mathcal{I}(I,n+1)\rightarrow[c]$ be a $c$-coloring. By definition, we have
\[
\mathcal{I}(I,n+1)=\bigcup_{i=0}^{W_{n+1}-1}\mathbb{T}^{iW_n\cdots W_1}\mathcal{I}(I,n).
\]
For simplicity, we denote the interval $\mathbb{T}^{iW_n\cdots W_1}\mathcal{I}(I,n)$ by $J_i$ for $i=0,1,\dots,$ $W_{n+1}-1$ (Recall that $|J_i|=W_n\cdots W_1$). Now, we define a coloring $\chi^*$ on the set $\{0,1,\dots,W_{n+1}-1\}$ as follows. For $0\leq i<j< W_{n+1}$, we put $\chi^*(i)=\chi^*(j)$ if for every $1\leq t\leq W_n\cdots W_1$, we have $\chi(J_i[t])= \chi(J_j[t])$. Since $\chi$ is a $c$-coloring, the number of colors for $\chi^*$ is $c^{W_n\cdots W_1}=c_n$. Thus, from $W_{n+1}=W(k,c_n)$, it follows that there exist $0\leq b_1<b_2<\cdots<b_k\leq W_{n+1}-1$ in the set $\{0,1,\dots,W_{n+1}-1\}$ with
\[
b_2-b_1=b_3-b_2=\cdots=b_k-b_{k-1}=:d^*\leq W_{n+1}
\]
such that $\chi^*(b_1)=\chi^*(b_2)=\cdots=\chi^*(b_k)$. Having
\[
J_{b_1}=\mathbb{T}^{b_1W_n\cdots W_1}\mathcal{I}(I,n)=\mathcal{I}(\mathbb{T}^{b_1W_n\cdots W_1}I,n),
\]
we can use the induction hypothesis for the interval $\mathbb{T}^{b_1W_n\cdots W_1}I$ with the restriction of the $c$-coloring $\chi$ to this interval and thus, obtain positive integers $a,d_1,\dots,d_n$ and a color $\gamma\in[c]$ such that the set
\[
\big{\{}a+j_1d_1+\cdots+j_nd_n|0\leq j_1,\dots,j_n\leq k-1\big{\}}
\]
is included in $J_{b_1}$ and is monochromatic with the color $\gamma$. Moreover, we have $d_1\leq W_1,\dots,d_n\leq W_1\cdots W_n$. Now, recall the definition of $d^*$ and notice that the intervals $J_i$'s have lengths equal to $W_1\cdots W_n$. Thus, we have
\begin{eqnarray*}
 \mathbb{T}^{jW_1\cdots W_nd^*}J_{b_1}&=&\mathbb{T}^{(jd^*+1)W_1\cdots W_n}\mathcal{I}(I,n)\\
                                      &=& \mathbb{T}^{(b_{1+j})W_1\cdots W_n}\mathcal{I}(I,n)
                                      = J_{b_{1+j}}
\end{eqnarray*}
for $j=0,\dots,k-1$.

So, for $1\leq i\leq W_1\cdots W_n$ and $j=0,\dots,k-1$, we get the relation
\[
J_{b_{1+j}}[i]=J_{b_1}[i]+jW_1\cdots W_nd^*.
\]
Now, let $a^*\in J_{b_1}$, then there is $1\leq i^*\leq W_n\cdots W_1$ such that $a^*=J_{b_1}[i^*]$. Therefore, for $j=0,\dots,k-1$, we have
\begin{eqnarray*}
\chi(a^*)=\chi(J_{b_1}[i^*])&=&\chi(J_{b_{1+j}}[i^*])\\
                            &=&\chi(J_{b_1}[i^*]+jW_1\cdots W_nd^*)\\
                            &=&\chi(a^*+jW_1\cdots W_nd^*).
\end{eqnarray*}
Putting this together with the induction hypothesis, we infer that the set
\[
\big{\{}a+j_1d_1+\cdots+j_nd_n+j_{n+1}W_1\cdots W_nd^*|0\leq j_1,\dots,j_n,j_{n+1}\leq k-1\big{\}}
\]
is monochromatic with the color $\gamma$ too. Now, set $d_{n+1}=W_1\cdots W_nd^*$. Recall that $d^*\leq W_{n+1}$ which implies $d_{n+1}\leq W_1\cdots W_{n+1}$  and observe that the above set is included in $J_{b_1}\cup\cdots\cup J_{b_k}$ which is included in $\mathcal{I}(I,n+1)$. So, the assertion is proved for $n+1$ and the proof is complete.
\end{proof}
\begin{theorem}\label{infvdw}
Let $k\geq 2$, and $c\geq 1$ be positive integers. Then, for any $c$-coloring of $\mathbb{N}$, there exist positive integers $d_1,d_2,\dots,d_n,\dots$ and a color $\gamma\in [c]$ such that for every $n\in\mathbb{N}$, there exists a positive integer $a_n$ such that the set
\[
\big{\{}a_n+j_1d_1+\cdots+j_nd_n|0\leq j_1,\dots,j_n\leq k-1\big{\}}
\]
is monochromatic with the color $\gamma$.
\end{theorem}
\begin{proof}
We fix $k\geq 2, c\geq 1$ and a coloring $\chi\colon\mathbb{N}\rightarrow[c]$. We inductively define sequences of intervals $\langle J_m; m\in\mathbb{N}\rangle$ and $\langle J_m^*; m\in\mathbb{N}\rangle$ as follows. Let $J_1=[1,W_1]$ and put $J_m^*=\big{[}1+\max J_m, W_1+\max J_m\big{]}$, $J_{m+1}=\mathcal{I}(J_m^*,m)$. According to Proposition \ref{first}, for every $m\in\mathbb{N}$, there exist positive integers $e_m,l_{m1},\dots,l_{mm}$ such that the set
\[
A_m:=\big{\{}e_m+j_1l_{m1}+\cdots+j_ml_{mm}|0\leq j_1,\dots,j_m\leq k-1\big{\}}\subseteq J_m
\]
is monochromatic. Moreover, for every $m\in\mathbb{N}$ we have
\[
l_{m1}\leq W_1,\dots,l_{mm}\leq W_1\cdots W_m.
\]
Since the set of colors is finite, there exist $\gamma\in[c]$ and an infinite subset $S_0\subseteq\mathbb{N}$ such that for each $m\in S_0$, the set $A_m$ is monochromatic with the color $\gamma$. Let $S_0= \{m_i|i\in\mathbb{N}\}$ with $m_i<m_{i+1}$, then, from the pigeonhole principle and
\[
l_{m_11},l_{m_21},l_{m_31},\ldots \leq W_1
\]
it follows that there exists an infinite subset $S_1\subseteq S_0$ such that for every $t_1,t_2\in S_1$, we have $l_{t_11}=l_{t_21}$.
Similarly, by iterative use of the pigeonhole principle, we obtain an infinite sequence of subsets
\[
S_0\supseteq S_1\supseteq\cdots \supseteq S_n\supseteq \cdots
\]
such that each $S_n$ is infinite and we have that if $t_1,t_2\in S_n$, then $l_{t_1n}=l_{t_2n}$. For $n\in\mathbb{N}$, let $s_n=\min S_n$ and set $a_n=e_{s_n},d_n=l_{s_nn}$. Now we verify that the set
\[
\big{\{}a_n+j_1d_1+\cdots+j_nd_n|0\leq j_1,\dots,j_n\leq k-1\big{\}}
\]
is monochromatic with the color $\gamma$. To see this, first observe that
\[
a_n+j_1d_1+\cdots+j_nd_n=e_{s_n}+j_1l_{s_11}+\cdots+j_nl_{s_nn}.
\]
As $s_n\in S_0$, we have $e_{s_n}+j_1l_{s_n1}+\cdots+j_nl_{s_nn}\in A_{s_n}$ and therefore
\[
\chi(e_{s_n}+j_1l_{s_n1}+\cdots+j_nl_{s_nn})=\gamma.
\]
Also from $s_n\in S_1,\dots,s_n\in S_{n-1}$, we get
\[
l_{s_n1}=l_{s_11},\dots,l_{s_n(n-1)}=l_{s_{n-1}(n-1)},
\]
which implies that
\begin{eqnarray*}
                      \chi(a_n+j_1d_1+\cdots+j_nd_n) &=&\chi(e_{s_n}+j_1l_{s_11}+\cdots+j_nl_{s_nn})\\
                                                     &=&\chi(e_{s_n}+j_1l_{s_n1}+\cdots+j_nl_{s_nn})=\gamma.
\end{eqnarray*}
This finishes the proof.
\end{proof}

By an easy modification of the arguments, we can actually prove a stronger theorem which is correspondent to Carlson's generalization \cite{carlson} of the Carlson-Simpson theorem. See also Theorem 3 of \cite{thomasshahram}.

\begin{theorem}
Let $c\geq 1$ be a positive integer. Let $2\leq k_1\leq k_2\leq\cdots\leq k_n\leq\cdots$ be a sequence of positive integers. Then for any $c$-coloring of $\mathbb{N}$, there exist positive integers $d_1,d_2,\dots$ and a color $\gamma\in [c]$ such that for every $n\in\mathbb{N}$, there exists a positive integer $a_n$ such that the set
\[
\big{\{}a_n+j_1d_1+\cdots+j_nd_n|0\leq j_1\leq k_1-1,\dots,0\leq  j_n\leq k_n-1\big{\}}
\]
is monochromatic with the color $\gamma$.
\end{theorem}
\begin{proof}
The proof proceeds exactly as in the proofs of Proposition \ref{first} and Theorem \ref{infvdw}, but this time, we work with $W_n=W(k_n,c_n)$ instead of $W_n=W(k,c_n)$.\end{proof}

\subsection*{Acknowledgment} We would like to thank the referee for careful reading of the paper and useful comments which led to many improvements in the presentation of the paper. The research of the author was in part supported by a grant from IPM (No. 00030403).

\bibliography{reference}
\bibliographystyle{plain}
\end{document}